\documentclass[english,11pt]{article}
\usepackage[T1]{fontenc}
\usepackage[latin9]{inputenc}
\usepackage{geometry}
\usepackage{color}
\usepackage{babel}
\usepackage{refstyle}
\usepackage{wrapfig}
\usepackage{units}
\usepackage{enumitem}
\usepackage{amsmath}
\usepackage{amsthm}
\usepackage{amssymb}
\usepackage{graphicx}
\usepackage[numbers]{natbib}
\usepackage[unicode=true,
 bookmarks=true,bookmarksnumbered=false,bookmarksopen=false,
 breaklinks=false,pdfborder={0 0 1},backref=page,colorlinks=true]
 {hyperref}
\hypersetup{pdftitle={Comparison of Risk Measures},
 pdfauthor={Alois Pichler},
 citecolor=blue, pdfstartview= FitH, pdfpagemode= UseNone}

\makeatletter

\AtBeginDocument{\providecommand\Floref[1]{\ref{Flo:#1}}}
\RS@ifundefined{subsecref}
  {\newref{subsec}{name = \RSsectxt}}
  {}
\RS@ifundefined{thmref}
  {\def\RSthmtxt{theorem~}\newref{thm}{name = \RSthmtxt}}
  {}
\RS@ifundefined{lemref}
  {\def\RSlemtxt{lemma~}\newref{lem}{name = \RSlemtxt}}
  {}

\theoremstyle{plain}
\newtheorem{thm}{\protect\theoremname}[section]
  \theoremstyle{plain}
  \newtheorem{lem}[thm]{\protect\lemmaname}
  \theoremstyle{definition}
  \newtheorem{defn}[thm]{\protect\definitionname}
  \theoremstyle{remark}
  \newtheorem{rem}[thm]{\protect\remarkname}
  \theoremstyle{plain}
  \newtheorem{prop}[thm]{\protect\propositionname}
  \theoremstyle{plain}
  \newtheorem{cor}[thm]{\protect\corollaryname}

\usepackage{dsfont}  	
\usepackage{lmodern}
\usepackage{microtype}   			
\usepackage{tikz}
\usepackage{doi}

\setlist[enumerate,1]{label=(\roman*)}

\DeclareMathOperator{\conv}{{conv}}

\DeclareMathOperator{\E}{{\mathbb E}}

\DeclareMathOperator{\var}{\mathsf {var}}

\DeclareMathOperator*{\argmin}{arg\,min}
\DeclareMathOperator*{\argmax}{arg\,max}

\newcommand{\HB}{\boldsymbol{\mathrm{HB}}}

\@ifundefined{showcaptionsetup}{}{%
 \PassOptionsToPackage{caption=false}{subfig}}
\usepackage{subfig}
\makeatother

  \providecommand{\corollaryname}{Corollary}
  \providecommand{\definitionname}{Definition}
  \providecommand{\lemmaname}{Lemma}
  \providecommand{\propositionname}{Proposition}
  \providecommand{\remarkname}{Remark}
\providecommand{\theoremname}{Theorem}

\begin{document}

\title{Geometry Of The Expected Value Set And\\
The Set-Valued Sample Mean Process }

\author{Alois Pichler\thanks{Faculty of Mathematics, Chemnitz University of Technology, Germany.
\protect \\
Contact: \protect\href{mailto:alois.pichler@mathematik.tu-chemnitz.de}{alois.pichler@mathematik.tu-chemnitz.de},
\protect\href{https://www.tu-chemnitz.de/mathematik/fima/}{https://www.tu-chemnitz.de/mathematik/fima/}}}
\maketitle
\begin{abstract}
The law of large numbers extends to random sets by employing Minkowski
addition. Above that, a central limit theorem is available for set-valued
random variables. The existing results use abstract isometries to
describe convergence of the sample mean process towards the limit,
the expected value set. These statements do not reveal the local geometry
and the relations of the sample mean and the expected value set, so
these descriptions are not entirely satisfactory in understanding
the limiting behavior of the sample mean process. This paper addresses
and describes the fluctuations of the sample average mean on the boundary
of the expectation set. 


\textbf{Classification:} 90C15, 26E25, 49J53, 28B20 
\end{abstract}

\section{Introduction}

\citet{Artstein1975} obtain an initial law of large numbers for random
sets. Given this result and the similarities of Minkowski addition
of sets with addition and multiplication for scalars it is natural
to ask for a central limit theorem for random sets. After some pioneering
work by \citet{Cressie1979}, \citet{Weil1982} succeeds in establishing
a reasonable result describing the distribution of the Pompeiu\textendash Hausdorff
distance between the sample average and the expected value set. The
result is based on an isometry between compact sets and their support
functions, which are continuous on some appropriate and adapted sphere
(cf.\ also \citet{Norkin2013} and \citet{Proske2003}; cf.\ \citet{Kuelbs1976}
for general difficulties). However, the Pompeiu\textendash Hausdorff
distance of random sets is just an $\mathbb{R}$\nobreakdash-valued
random variable and its distribution is on the real line. But how
do these sample averages, as sets in $\mathbb{R}^{d}$, converge locally?
We address this question for selected points at the boundary of the
expected value set. 

This paper elaborates local features of set-valued convergence of
sample means and the distribution of particular selections is in focus
of our interest. To develop the intuitive understanding we specify
and restrict ourselves occasionally to a discrete setting chosen in
\citet{Cressie1979}; this situation is natural for set-valued risk
functionals in mathematical finance as well.

\paragraph{Outline of the paper. }

We introduce the expectation and the Pompeiu\textendash Hausdorff
distance in Section~\ref{sec:Preliminaries}. Of particular interest
are the boundary points of the expected value. We classify the boundary
points in Section~\ref{sec:4} and discuss relations between boundary
points of the expected value set and corresponding points of the sample
means. Section~\ref{sec:LLN} addresses the Law of Large Numbers
and Section~\ref{sec:CLT} the Central Limit Theorem. These sections
contain our main results, which describe convergence of sample means
relative to particular points on the boundary. Section~\ref{sec:Summary}
concludes and summarizes the results. 

\section{\label{sec:Preliminaries}Mathematical setting}

We work in $\mathbb{R}^{d}$ with norm $\left\Vert \cdot\right\Vert $.
We denote this space by $X:=\left(\mathbb{R}^{d},\left\Vert \cdot\right\Vert \right)$,
its dual by $X^{*}:=\left(\mathbb{R}^{d},\left\Vert \cdot\right\Vert _{*}\right)$
and the unit sphere in the dual by $S^{d-1}:=\left\{ x:\,\left\Vert x\right\Vert _{*}=1\right\} $.
The Minkowski sum (also known as dilation) of two subsets $A$ and
$B$ of $\mathbb{R}^{d}$ is $A+B:=\left\{ a+b\colon a\in A,\,b\in B\right\} $
and the product with a scalar $p$ is $p\cdot A:=\left\{ p\cdot a\colon a\in A\right\} $.
 We denote the convex hull of a set $A$ by $\conv A$ and its topological
closure by $\overline{\conv}A$.

\paragraph{Pompeiu\textendash Hausdorff Distance.}

The appropriate distance on $\mathfrak{C}_{d}$, the set of compact
subsets of $\mathbb{R}^{d}$, is the Pompeiu\textendash Haus\-dorff
distance. For this define the point-to-set distance as $d(a,B):=\inf_{b\in B}\left\Vert b-a\right\Vert $.
The \emph{deviation }of the set $A$ from the set $B$ is $\mathbb{D}(A,B):=\sup_{a\in A}d\left(a,B\right)$.\footnote{An equivalent definition is $\mathbb{D}(A,B):=\inf\left\{ \varepsilon>0:A\subset B+Ball_{\varepsilon}(0)\right\} $;
here, $B_{\varepsilon}:=B+Ball_{\varepsilon}(0)$ is often called
$\varepsilon$\nobreakdash-fattening, or $\varepsilon$\nobreakdash-enlargement
of $B$.} (Some references call $\mathbb{D}(A,B)$ the excess of $A$ over
$B$, cf.\ \citet{Hess2002}.) The \emph{Pompeiu\textendash Hausdorff
distance }is $\mathbb{H}\left(A,B\right):=\max\left\{ \mathbb{D}\left(A,B\right),\,\mathbb{D}\left(B,A\right)\right\} $,
cf.\ also \citet{WetsRockafellar97}. 

Note that $\mathbb{D}(A,B)=0$ iff $A$ is contained in the topological
closure of $B$, $A\subseteq\overline{B}$, and $\mathbb{H}\left(A,B\right)=0$
iff $\overline{A}=\overline{B}$; moreover $\mathbb{H}\left(A,B\right)=\mathbb{H}\left(\overline{A},B\right)$.

If $\overline{A}$ and $\overline{B}$ are  compact and convex then
it is enough to consider their boundaries $\partial A$ and $\partial B$,
as we have in addition that $\mathbb{H}\left(A,B\right)=\mathbb{H}\left(\partial A,\partial B\right)$
(cf.~\citet{Wills2007}). In this case we have 
\begin{equation}
\mathbb{H}(A,B)=\left\Vert b-a\right\Vert \label{eq:HausdorffBoundary}
\end{equation}
for some $a\in\partial A$ and $b\in\partial B$. 
\begin{lem}[\citet{Castaing1977}]
The deviation $\mathbb{D}$ and the Pompeiu\textendash Hausdorff
distance $\mathbb{H}$ satisfy the triangle inequality, $\mathbb{D}\left(A,C\right)\le\mathbb{D}\left(A,B\right)+\mathbb{D}\left(B,C\right)$
and $\mathbb{H}\left(A,C\right)\le\mathbb{H}\left(A,B\right)+\mathbb{H}\left(B,C\right)$.
For a Polish space $\left(X,d\right)$ the space $\left(\mathfrak{C},\,\mathbb{H}\right)$,
where $\mathfrak{C}$ is the set of all nonempty, compact and convex
subsets of $X$, is a Polish space again (i.e., a complete, separable
and metric space).
\end{lem}

By the preceding lemma $(\mathfrak{C}_{d},\mathbb{H})$, the nonempty
compact subsets of $\mathbb{R}^{d}$ endowed with the Pompeiu\textendash Hausdorff
distance $\mathbb{H}$, is a measurable space. In what follows we
equip $\mathfrak{C}_{d}$ with the sigma algebra of its Borel subsets
generated by the family of closed subsets of $\mathfrak{C}_{d}$.

\subsection{\label{sec:Expectation}Expectation}

We consider a set-valued random variable $Y\colon\Omega\rightrightarrows\mathbb{R}^{d}$
(commonly random sets) on some complete probability space $(\Omega,\mathcal{F},P)$.
Throughout the paper we assume that the set-valued random variable
$Y:\Omega\rightrightarrows\mathbb{R}^{d}$ is compact-valued and measurable,
i.e., the associated map $Y\colon\Omega\to\left(\mathfrak{C}_{d},\mathbb{H}\right)$
is measurable. 
\begin{defn}[{Expectation, cf.\ \citet[Definition~1.12]{Molchanov}}]
\label{def:Expectation}The expectation $\E Y$ of a set-valued random
variable $Y\colon\Omega\rightrightarrows\mathbb{R}^{d}$ is the collection
\begin{equation}
\E Y:=\left\{ \int_{\Omega}\mathbf{y}\mathrm{d}P\colon\mathbf{y}(\cdot)\text{ an integrable selection of }Y\right\} \subseteq\mathbb{R}^{d};\label{eq:Expectation}
\end{equation}
a function $\mathbf{y}\colon\Omega\rightarrow\mathbb{R}^{d}$ is an
integrable selection of $Y$ if $\mathbf{y}(\omega)\in Y(\omega)$
for $P$\nobreakdash-almost every $\omega\in\Omega$ and $\mathbf{y}(\cdot)$
is $P$\nobreakdash-integrable, i.e., $\int\left\Vert \mathbf{y}(\omega)\right\Vert P(\mathrm{d}\omega)<\infty$.
The expectation~(\ref{eq:Expectation}) (also Aumann expectation)
is often denoted $\E Y=\int Y\mathrm{d}P$ as well. 
\end{defn}

\subsubsection*{Atomic versus non-atomic probability spaces}

Consider a set-valued random variable $Y$ defined on an \emph{atomic}
space $(\Omega,\mathcal{F},P)$ such that 
\begin{equation}
P(Y=K_{1})=p_{1},\,P(Y=K_{2})=p_{2},\dots\text{ and }P(Y=K_{J})=p_{J}\label{eq:discrete}
\end{equation}
for finitely many sets $\left(K_{j}\right)_{j=1}^{J}$ with weights
$p_{j}>0$, $\sum_{j=1}^{J}p_{j}=1$. From the definition of the expected
value~(\ref{eq:Expectation}) it is evident that 
\[
\E Y=\int Y\mathrm{d}P=\sum_{j=1}^{J}p_{j}K_{j}
\]
(cf.\ \citet{Cressie1979} and Figure~\ref{fig:ALineB} for illustration).
$\E Y$ is moreover compact, provided that  all $K_{j}$ are compact.
\begin{figure}[t]
\centering{}\begin{tikzpicture}[scale= 0.6]

\node[blue!50!red] at (5.5,1.9) {$\frac 1 2(A+B)$};

\filldraw[blue!50!red] (2.5,0.5)--(3.5,1.5)--(6.5,1.5)--(5.5,0.5)--cycle;
\draw[red, very thick] (3,0)--(9,0);	\node[red] at (8,0.4) {$A$};
\draw[black!10!blue, very thick] (2,1)--(4,3);	\node[black!10!blue] at (3,2.6) {$B$};
\end{tikzpicture}\hfill{}\begin{tikzpicture}[scale= 0.6]

\draw[black!10!red, very thick](-3,1) circle (.5); \node[black!10!red]  at (-3.5,0.3) {$A$};
\draw[blue, very thick](3,0) circle (1.5); \node[blue] at (4.4,-1.2){$B$};

\filldraw[even odd rule, blue!50!red](0,0.5) circle (1.0) (0,0.5) circle (.5);
\node[blue!50!red]  at (-0.1,-1.3) {$\frac 1 2(A+B)$};
\end{tikzpicture}\hfill{}\begin{tikzpicture}[xscale= 0.6, yscale=0.43]

\node[blue!50!red] at (4,2.0) {$\frac 1 2(A+B)$};
\draw[red, very thick] (2,-2)--(0,0)--(2,2); \node[red] at (0,-1) {$A$};
\draw[blue, very thick] (6,-1)--(6,1); \node[blue] at (5.5,-1) {$B$};
\filldraw[blue!50!red] (4,-1.5)--(3,-0.5)--(3,0.5)--(4,1.5)--(4,0.5)--(3,-0.5)--(3,0.5) --(4,-.5)--(4,-1.5);

\end{tikzpicture}\caption{Two sets $A$ and $B$ (lines), their mean $\frac{1}{2}(A+B)$ is
an area in $\mathbb{R}^{2}$\label{fig:ALineB}}
\end{figure}
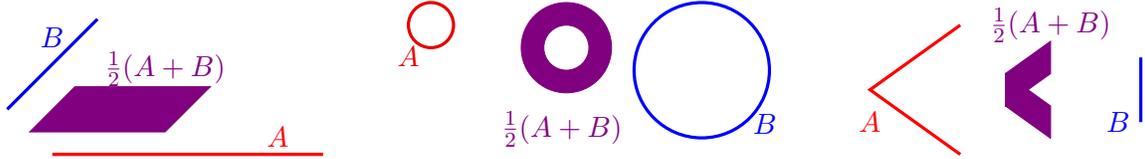

Note that $\E Y$ is not necessarily convex. The expectation $\E Y$
is convex, provided that all sets $K_{j}$ are convex, as $\E\conv Y=\sum_{j=1}^{J}p_{j}\conv K_{j}=\conv\sum_{j=1}^{J}p_{j}K_{j}=\conv\E Y$,
where $\conv A$ denotes the convex hull of the set $A$.

The situation notably changes for non-atomic probability spaces. Aumann's
Theorem (cf.\ \citet[Theorem~2]{Aumann1965}) ensures that $\E Y$
is non-empty, compact \emph{and convex}, provided that $P$ does not
have atoms and there is an integrable random variable $h(\cdot)$,
called an \emph{envelope function}, such that 
\begin{equation}
\left\Vert Y(\omega)\right\Vert :=\sup_{y\in Y(\omega)}\left\Vert y\right\Vert \le h(\omega).\label{eq:Envelope}
\end{equation}

Unless stated differently we shall assume the standard, non-atomic
probability space in what follows. Further, the random set $Y$ is
assumed to be compact, convex valued and integrably bounded, i.e.,
$\left\Vert Y(\cdot)\right\Vert $ is measurable and $\int\left\Vert Y(\cdot)\right\Vert \mathrm{d}P<\infty$
(cf.\ \citet[Definition~1.11]{Molchanov}): Section~\ref{subsec:Convexification}
below outlines why this setting is not an essential restriction in
investigating the law of large numbers and the central limit theorem.
As well, the chosen setting insures that the expectation $\E Y$ defined
in~(\ref{eq:Expectation}) is closed (cf.\ \citet[Theorem~1.24]{Molchanov}).

\subsection{Support function }

The support function of a set $A\subseteq X$ is
\begin{equation}
s_{A}\big(x^{*}\big):=\sup_{a\in A}x^{*}(a),\label{eq:supportFctn}
\end{equation}
where $x^{*}\in X^{*}$ is from the dual $X^{*}=\left(\mathbb{R}^{d},\left\Vert \cdot\right\Vert \right)^{*}=\left(\mathbb{R}^{d},\left\Vert \cdot\right\Vert _{*}\right)$. 

By the Fenchel\textendash Moreau-duality theorem (cf.\ \citet{Rockafellar1974})
we  have the relation 
\[
\overline{\conv}A=\left\{ s_{A}^{*}<\infty\right\} ,
\]
where $s_{A}^{*}(a):=\sup_{x^{*}\in X^{*}}x^{*}(a)-s_{A}(x^{*})$
is the convex conjugate of $s_{A}$. The correspondence $A\mapsto s_{A}$
is one-to-one (injective) between convex, compact sets $A\in\mathfrak{C}_{d}$
and finite valued convex positively homogeneous functions on $\mathbb{R}^{d}$
and satisfies the isometry 
\begin{equation}
\sup_{a\in A}\left\Vert a\right\Vert =\sup_{a\in A}\sup_{\left\Vert x^{*}\right\Vert _{*}=1}x^{*}(a)=\sup_{\left\Vert x^{*}\right\Vert _{*}=1}s_{A}(x^{*})=\left\Vert s_{A}\right\Vert _{\infty},\label{eq:Norm}
\end{equation}
where the norm on the space $C(S^{d-1})$ of bounded and continuous
functions defined on the unit sphere in the dual space 
\begin{equation}
S^{d-1}:=\left\{ x^{*}\in\mathbb{R}^{d}:\left\Vert x^{*}\right\Vert _{*}=1\right\} =\partial B_{X^{*}}\label{eq:Sphere}
\end{equation}
is $\left\Vert f\right\Vert _{\infty}:=\sup_{x\in S^{d-1}}\left|f(x)\right|$.

 As the support function is positively homogeneous ($s_{A}(\lambda x^{*})=\lambda s_{A}(x^{*})$
for $\lambda>0$), one may restrict $s_{A}$ to the unit sphere of
the dual without losing information (cf.~(\ref{eq:Sphere})). The
mapping $K\mapsto\left.s_{K}\right|_{\partial B_{X^{*}}}$ (the restriction
to the sphere $S^{d-1}$) is an \emph{isometric isomorphism }from
$\mathfrak{C}_{d}$, the convex, compact subsets of \emph{$\mathbb{R}^{d}$
onto }$C\left(S^{d-1}\right)$, the Banach space of continuous functions
endowed with the norm $\left\Vert f\right\Vert _{\infty}=\sup_{s\in\partial B_{X^{*}}}\left|f(s)\right|$
on the compact set $S^{d-1}=\partial B_{X^{*}}$ by~(\ref{eq:Norm}). 

\subsection{\label{sec:LocalDescription}Tangent planes}

The subdifferential of an $\mathbb{R}$-valued function $f:X^{*}\rightarrow\mathbb{R}$
at a point $x^{*}\in X^{*}$ is the set 
\[
\partial f\left(x^{*}\right):=\left\{ u\in X\colon f\left(z^{*}\right)-f\left(x^{*}\right)\ge z^{*}(u)-x^{*}(u)\text{ for all }z^{*}\in X^{*}\right\} \subseteq X.
\]
The subdifferential $\partial f\left(x^{*}\right)$ is  a convex subset
of $X$, so $\partial f$ is a set-valued mapping, 
\begin{align*}
\partial f\colon X^{*} & \rightrightarrows X\\
x^{*} & \mapsto\partial f\left(x^{*}\right).
\end{align*}

With the subdifferential at hand we have the following characterization
of the subdifferential of the support function $s_{A}$ of a set $A$
(the bipolar theorem for indicator functions), which will turn out
useful in investigating the expected value set. 
\begin{lem}
\label{thm:SupArgMax}The support function $s_{A}$ has the subdifferential
\begin{equation}
\partial s_{A}\left(x^{*}\right)=\argmax_{\overline{\conv}A}x^{*},\label{eq:Argmax}
\end{equation}
where $x^{*}\in X^{*}$ and 
\[
\argmax_{D}f:=\argmax\left\{ f(d):\,d\in D\right\} =\left\{ x\in D:\,f(x)\ge f(x^{\prime})\text{ for all }x^{\prime}\in D\right\} .
\]
Moreover, $\partial s_{A}\left(x^{*}\right)\subseteq\partial\conv A$
for every $x^{*}\in X^{*}$. 
\end{lem}

\begin{proof}
Note first that 
\[
s_{A}=s_{\overline{\conv}A}.
\]
Indeed, it is evident that $s_{A}\le s_{\overline{\conv}A}$ by definition;
for the converse choose $a=\sum_{i=1}^{n}\lambda_{i}a_{i}\in\conv A$
with $a_{i}\in A$, $\lambda_{i}>0$, $i=1,\dots,n$ and $\sum_{i=1}^{n}\lambda_{i}=1$
so that $s_{\overline{\conv}A}(x^{*})<x^{*}(a)+\varepsilon$. By linearity
we also have that $s_{\overline{\conv}A}(x^{*})<x^{*}(a_{i^{*}})+\varepsilon$,
where $i^{*}$ is chosen so that $x^{*}(a_{i^{*}})\ge x^{*}(a_{i})$
for all $i=1,\dots,n$.

We deduce then from \citet[Corollary 23.5.3]{Rockafellar1970} that
$\argmax_{a\in\overline{\conv}A}x^{*}(a)=\partial s_{\overline{\conv}A}\left(x^{*}\right)$,
so that the assertion follows.
\end{proof}
\begin{rem}[Hörmander's theorem, cf.\ \citet{Hoermander1955}]
The concepts of Hausdorff distance and support functions introduced
above link to a nice ensemble, as the deviation $\mathbb{D}$ can
also be states as $\mathbb{D}(A,C)=\sup_{a\in A}\inf_{c\in C}\sup_{\left\Vert x^{*}\right\Vert _{*}\le1}x^{*}(a-c)$.
It follows from the max-min inequality that 
\begin{align}
\mathbb{D}(A,C) & =\sup_{a\in A}\inf_{c\in C}\sup_{\left\Vert x^{*}\right\Vert _{*}\le1}x^{*}(a-c)\ge\sup_{a\in A}\sup_{\left\Vert x^{*}\right\Vert _{*}\le1}\inf_{c\in C}x^{*}(a-c)\label{eq:5-5}\\
 & =\sup_{a\in A}\sup_{\left\Vert x^{*}\right\Vert _{*}\le1}\left\{ x^{*}(a)-\sup_{c\in C}x^{*}(c)\right\} =\sup_{\left\Vert x^{*}\right\Vert _{*}\le1}\left\{ s_{A}(x^{*})-s_{C}(x^{*})\right\} .\nonumber 
\end{align}
Assuming that $A$ and $C$ are convex it follows from compactness
of the dual ball and the minimax theorem (\citet[Theorem 2]{Fan1953})
that equality holds in (\ref{eq:5-5}), hence 
\[
\mathbb{D}(\conv A,\,\conv C)=\sup_{\left\Vert x^{*}\right\Vert _{*}\le1}\left\{ s_{A}\left(x^{*}\right)-s_{C}\left(x^{*}\right)\right\} ,
\]
the Pompei\textendash Hausdorff distance thus is
\begin{equation}
\mathbb{H}\left(\conv A,\,\conv C\right)=\sup_{\left\Vert x^{*}\right\Vert _{*}\le1}\left|s_{A}\left(x^{*}\right)-s_{C}\left(x^{*}\right)\right|,\label{eq:Hausdorff}
\end{equation}
expressed in terms of seminorms. These observations and~(\ref{eq:Norm})
convincingly relate the Pompeiu\textendash Hausdorff distance with
Min\-kowski addition of convex sets.
\end{rem}

It follows from the preceding discussion and remarks that for relatively
compact sets $A$ and $C$ there are $a\in\partial A$, $c\in\partial C$
and $\left\Vert x^{*}\right\Vert _{*}\le1$ such that $\mathbb{D}\left(A,C\right)=\left\Vert c-a\right\Vert =x^{*}(a-c)$.
$x^{*}$ is an outer normal for both sets, $\conv A$ and $\conv C$.

\section{\label{sec:4}The relative boundary of the expected value}

We shall use tangent planes to investigate the convex expected value
set. To this end let $f\in X^{*}$ be a linear functional. By Aumann's
Theorem, the set-valued mapping 
\begin{equation}
\omega\mapsto\partial s_{Y(\omega)}\left(f\right)\subseteq\mathbb{R}^{d}\label{eq:selection1}
\end{equation}
is measurable and $\E\partial s_{Y}\left(f\right)=\int\partial s_{Y(\omega)}\left(f\right)P(\mathrm{d}\omega)$
is non-empty, compact and convex (cf.\ \citet[Theorem~2]{Aumann1965}).
We continue with a characterization of this expected value. For a
related result on the interchangeability of the differentiation $\partial$
and expectation $\E$ we refer to \citet{Rockafellar1982}.
\begin{prop}
\label{thm:4}Suppose that $f\in X^{*}$. Then 
\begin{equation}
\E\partial s_{Y}(f)=\partial s_{\E Y}(f)\subseteq\partial\E Y.\label{eq:11}
\end{equation}
\end{prop}

\begin{proof}
Let $e\in\E\partial s_{Y}\left(f\right)$ have the representation
$e=\int\mathbf{e}\mathrm{d}P$ and recall from Lemma~\ref{thm:SupArgMax}
that $\mathbf{e}(\omega)\in\partial s_{Y(\omega)}\left(f\right)=\argmax_{y\in Y(\omega)}f(y)$.
Note as well that $\mathbf{e}(\cdot)$ can be chosen measurable by
the Kuratowski and Ryll\textendash Nardzewski measurable selection
theorem, cf.\ \citet[Volume~II, page~36]{Bogachev2007} or \citet[Theorem~2]{Aumann1965}.
Hence, for every measurable $\mathbf{y}$ with $\mathbf{y}(\cdot)\in Y(\cdot)$
we have that $f(\mathbf{e}(\omega))\ge f\left(\mathbf{y}(\omega)\right)$.
Define $y:=\int\mathbf{y}\mathrm{d}P$, then  
\begin{align*}
f(e) & =f\left(\int\mathbf{e}\mathrm{d}P\right)=\int f(\mathbf{e}(\omega))P(\mathrm{d}\omega)\\
 & \ge\int f\left(\mathbf{y}(\omega)\right)P(\mathrm{d}\omega)=f\left(\int\mathbf{y}\mathrm{d}P\right)=f(y)
\end{align*}
by linearity of $f$ for every measurable selection $\mathbf{y}$.
Hence, $e\in\argmax_{y\in\E Y}f(y)=\partial s_{\E Y}\left(f\right)$
by~(\ref{eq:Argmax}), which is the inclusion $\subseteq$ of set-equality
in~(\ref{eq:11}). 

For the converse assume that $e\in\partial s_{\E Y}\left(f\right)\backslash\E\partial s_{Y}\left(f\right)$.
As $\E\partial s_{Y}\left(f\right)$ is convex and compact it follows
from the separation theorem that there is an $\alpha\in\mathbb{R}$
such that 
\begin{equation}
f(e)=\int f(\mathbf{e})\mathrm{d}P>\alpha>\int f(\mathbf{y})\mathrm{d}P\label{eq:10}
\end{equation}
for every measurable $\mathbf{y}(\omega)\in\partial s_{Y(\omega)}(f)$.
Notice that $\mathbf{y}(\omega)\in\partial s_{Y(\omega)}(f)=\argmax_{y\in Y(\omega)}f(y)\subseteq Y(\omega)$.
By the particular choice of $e$ it follows for $y:=\int\mathbf{y}\mathrm{d}P\in\E Y$
that $f(e)>f(y)$. 

However, by~(\ref{eq:10}), on a set of strictly positive $P$\nobreakdash-measure
we have that 
\[
P\bigl(\left\{ \omega\colon f(\mathbf{e}(\omega))>f\left(\mathbf{y}(\omega)\right)\right\} \bigr)>0.
\]
On this set $\mathbf{y}(\omega)\notin\argmax_{y\in Y(\omega)}f(y)=\partial s_{Y(\omega)}\left(f\right)$,
because $f(\mathbf{e}(\omega))>f(\mathbf{y}(\omega))$. This is a
contradiction, because $y\in\partial s_{Y(\omega)}\left(f\right)=\argmax_{Y(\omega)}(f)$
$P$\nobreakdash-almost everywhere. 

The remaining inclusion follows from Lemma~\ref{thm:SupArgMax}.
\end{proof}
\bigskip{}
We deduce from the previous proposition that the set-valued subdifferential
$\partial$ and the set-valued expectation $\E$ commute. Moreover,
the set-valued subdifferential of the support function basically is
its $\argmax$\nobreakdash-set, which is an element from the boundary
of the respective set. This is another hint that the \emph{boundary}
$\partial\E Y$ plays a central role, which we intend to investigate
in more detail in what follows.

\subsection{Extreme and exposed points}

It will be convenient to classify the boundary points of the convex
set $\E Y$ based on the following definitions. 
\begin{defn}[Extreme points, exposed points]
\label{def:Strict}Let $K$ be a convex set.
\begin{enumerate}
\item \label{enu:extreme}$k\in K$ is an \emph{extreme point }if $k=\frac{1}{2}\left(k_{1}+k_{2}\right)$
for $k_{1}\in K$ and $k_{2}\in K$ implies that $k_{1}=k_{2}$.
\item $k\in K$ is an \emph{exposed point }if there is a linear, continuous
functional $f$ such that $f(k)>f(x)$ for all $x\in K\backslash\left\{ k\right\} $.
$f$ is said to \emph{expose }$k\in K$. The collection of all exposed
points of the set $K$ is denoted b\textcolor{blue}{y} $\exp K$. 
\item \label{enu:strict}$K$ is strictly convex, if $\left\{ (1-\lambda)k_{0}+\lambda k_{1}\colon0<\lambda<1\right\} \subseteq\overset{\circ}{K}$,
the interior of $K$, whenever $k_{0},\,k_{1}\in K$ and $k_{0}\neq k_{1}$.
\end{enumerate}
\end{defn}

\begin{rem}
\label{remStrictlyConvex}The point $e$ in Figure~\ref{fig:FacetHB-1}
on page~\pageref{fig:FacetHB-1} is extreme, but not exposed. 
\end{rem}

\begin{rem}[Boundary points of strictly convex sets are exposed]
If $K$ is strictly convex, then every boundary point $k\in\partial K$
is exposed. Indeed, let $f$ be a linear, separating functional such
that $f(k)>f(x)$ for all $x\in\overset{\circ}{K}$ ($f$ exists by
the Hahn\textendash Banach theorem). Suppose there were another $\tilde{k}\in K$
such that $f(\tilde{k})=f(k)$. As $\tfrac{1}{2}\left(k+\tilde{k}\right)\in\overset{\circ}{K}$
by assumption it follows that $f(k)>f\left(\tfrac{1}{2}\left(k+\tilde{k}\right)\right)=\tfrac{1}{2}\left(f\left(k\right)+f(\tilde{k})\right)=f(k)$,
which is a contradiction. Hence $f$ exposes $k$ and $\partial K=\exp K$. 
\end{rem}

\subsection{The boundary of $\E Y$}

We return to the geometry of $\E Y$ and discuss exposed points of
$\E Y$ first. The next theorem elaborates that exposed points of
$\E Y$ are comparably seldom, as being exposed in $\E Y$ means that
the exposing functional exposes points of $Y(\omega)$ for almost
every $\omega\in\Omega$.
\begin{thm}
\label{thm:6}Let $e$ be an exposed point of $\E Y$, exposed by
a linear functional $f$. Then $f$ exposes a single point of $Y(\omega)$
$P$\nobreakdash-almost everywhere.

Moreover, there is just a single measurable selection $\mathbf{e}$
such that $e=\int\mathbf{e}\mathrm{d}P$, i.e., $\mathbf{e}$ is $P$\nobreakdash-almost
everywhere unique.
\end{thm}

\begin{proof}
Let the exposed point $e\in\E Y$ have the representation $e=\int\mathbf{e}_{1}\mathrm{d}P$,
where $\mathbf{e}_{1}(\cdot)\in Y(\cdot)$ is a measurable selection
according~(\ref{eq:Expectation}). By definition of an exposed point
$\left\{ e\right\} =\argmax_{y\in\E Y}f(y)$ and by Theorem~\ref{thm:4}
we have that $e\in\E\partial s_{Y}\left(f\right)$, which means that
$\mathbf{e}_{1}(\omega)\in\argmax_{y\in Y(\omega)}f(y)$ $P$\nobreakdash-a.e.
If this representation were not unique, then there is another measurable
selection $\mathbf{e}_{2}(\omega)\in\argmax_{y\in Y(\omega)}f(y)$
with $e=\int\mathbf{e}_{2}\mathrm{d}P$ and $P\left(\mathbf{e}_{1}\neq\mathbf{e}_{2}\right)>0$.
In this situation there is a linear functional $\ell$ such that $P\bigl(\ell\left(\mathbf{e}_{1}\right)\neq\ell\left(\mathbf{e}_{2}\right)\bigr)>0$.
Define the random variable $\tilde{\mathbf{e}}_{1}:=\begin{cases}
\mathbf{e}_{1} & \text{if }\ell\left(\mathbf{e}_{1}\right)\le\ell\left(\mathbf{e}_{2}\right),\\
\mathbf{e}_{2} & \text{if }\ell\left(\mathbf{e}_{1}\right)>\ell\left(\mathbf{e}_{2}\right)
\end{cases}$ and $\tilde{\mathbf{e}}_{2}:=\begin{cases}
\mathbf{e}_{2} & \text{if }\ell\left(\mathbf{e}_{1}\right)\le\ell\left(\mathbf{e}_{2}\right),\\
\mathbf{e}_{1} & \text{if }\ell\left(\mathbf{e}_{1}\right)>\ell\left(\mathbf{e}_{2}\right)
\end{cases}$. Notice that $\ell\left(\tilde{\mathbf{e}}_{1}\right)\le\ell\left(\tilde{\mathbf{e}}_{2}\right)$,
and $P\bigl(\ell\left(\tilde{\mathbf{e}}_{1}\right)<\ell\left(\tilde{\mathbf{e}}_{2}\right)\bigr)>0$.
Hence $e=\frac{1}{2}\left(\int\tilde{\mathbf{e}}_{1}\mathrm{d}P+\int\tilde{\mathbf{e}}_{2}\mathrm{d}P\right)$
and $\int\ell\left(\tilde{\mathbf{e}}_{1}\right)\mathrm{d}P<\int\ell\left(\tilde{\mathbf{e}}_{2}\right)\mathrm{d}P$,
and by linearity of $\ell$ thus $e_{1}:=\int\tilde{\mathbf{e}}_{1}\mathrm{d}P\neq\int\tilde{\mathbf{e}}_{2}\mathrm{d}P=:e_{2}$.
This is a contradiction, because $f$ can only expose one unique point
$e\in\E Y$. This proves the second assertion.

The first assertion follows, as $\mathbf{e}(\cdot)\in\argmax_{y\in Y(\cdot)}f(y)$
is $P$\nobreakdash-almost everywhere unique by the second, and $f$
thus exposes $\mathbf{e}(\omega)\in Y(\omega)$. 
\end{proof}
We note the contrapositive statement of the previous theorem, Theorem~\ref{thm:6}. 
\begin{cor}
Suppose that the linear functional $f$ does not expose a point from
$Y(\cdot)$ almost everywhere. Then $f$ does not expose a point of
$\E Y$.
\end{cor}

The statement of the preceding theorem of course holds true for discrete
distributions as in~(\ref{eq:discrete}), although the proof simplifies
significantly. We record the next lemma to emphasize that the $\argmax$\nobreakdash-set
of the sample means \emph{in addition} is the sample mean of the respective
$\argmax$\nobreakdash-sets\textemdash an observation of further
importance for the sample mean process discussed later.

For the next lemma see also \citep[Theorem~2.8.7]{Zalinescu} or \citet[Theorem~3.5.8]{BotGradWanka2009}.
\begin{lem}
\label{Lem:argmax}Let $Y$ be a random map according (\ref{eq:discrete})
with compact and convex outcome and $f\in X^{*}$. Then 
\begin{equation}
\argmax_{k\in\E Y}f(k)=\sum_{j=1}^{J}p_{j}\argmax_{k\in K_{j}}f(k)=\E\argmax_{k\in Y}f(k).\label{eq:argmax}
\end{equation}
 Moreover 
\begin{equation}
\argmax_{k\in\frac{1}{N}\sum_{i=1}^{N}K_{i}}f(k)=\frac{1}{N}\sum_{i=1}^{N}\argmax_{k\in K_{i}}f(k)\qquad(\omega\in\Omega)\label{eq:argmaxSn}
\end{equation}
for any sequence of compact and convex sets $\left(K_{i}(\omega)\right)_{i=1}^{N}$. 
\end{lem}

\begin{proof}
As for (\ref{eq:argmax}) fix $k\in\argmax_{\E Y}f\subseteq\E Y=\sum_{j=1}^{J}p_{j}K_{j}$,
which may be written as $k=\sum_{j}p_{j}k_{j}$ with $k_{j}\in K_{j}$.
For any $y_{j}\in K_{j}$, $y:=\sum_{j}p_{j}y_{j}\in\E Y$. By linearity
and $f$\nobreakdash-maximality of $k$, 
\[
\sum_{j}p_{j}f\left(k_{j}\right)=f(k)\ge f(y)=\sum_{j}p_{j}f\left(y_{j}\right)
\]
for any $y_{j}\in K_{j}$, hence $k_{j}\in\argmax_{K_{j}}f$. This
proves that $\argmax_{\E Y}f\subseteq\sum_{j}p_{j}\argmax_{K_{j}}f$.

Conversely observe first that any $k\in\sum_{j}p_{j}\argmax_{K_{j}}f$
has a representation $k=\sum_{j}p_{j}k_{j}$ for $k_{j}\in\argmax_{K_{j}}f$.
As $k_{j}\in\argmax_{K_{j}}f\subseteq K_{j}$ it is thus obvious that
$k=\sum_{j}p_{j}k_{j}\in\sum_{j}p_{j}K_{j}=\E Y$. Now pick any $y\in\E Y$
with representation $y=\sum_{j}p_{j}y_{j}$ and $y_{j}\in K_{j}$.
By linearity and maximality of~$k_{j}$, 
\[
f(k)=\sum_{j}p_{j}f(k_{j})\ge\sum_{j}p_{j}f(y_{j})=f(y),
\]
hence $k\in\argmax_{\E Y}f$, that is $\sum_{j}p_{j}\argmax_{K_{j}}f\subseteq\argmax_{\E Y}f$.
Summarizing the inclusions, $\argmax_{\E Y}f=\sum_{j}p_{j}\argmax_{K_{j}}f$.
By $\E\argmax_{Y}f=\sum_{j}p_{j}\argmax_{K_{j}}f$ the assertion finally
follows.

Equation~(\ref{eq:argmaxSn}) verifies along the same lines as the
proof for~(\ref{eq:argmax}), but $p_{j}$ replaced by~$\frac{1}{N}$. 
\end{proof}
The following two theorems address the other properties introduced
in Definition~\ref{def:Strict}, which are strict convexity (Theorem~\ref{thm:Strict}
below) and extreme points (Theorem~\ref{thm:Extreme}).
\begin{thm}
\label{thm:Strict}Let $Y$ be strictly convex almost surely. Then
$\E Y$ is strictly convex as well.
\end{thm}

\begin{proof}
Let $k_{1}$, $k_{2}\in\E Y$ be chosen so that $k_{1}\not=k_{2}$
and let $\mathbf{k}_{1}$ and $\mathbf{k}_{2}$ be measurable selections
so that $k_{1}=\int\mathbf{k}_{1}\mathrm{d}P$ and $k_{2}=\int\mathbf{k}_{2}\mathrm{d}P$.
Note, that there is a measurable set $A$ with $P(A)>0$ and $A\subset\left\{ \left\Vert \mathbf{k}_{1}-\mathbf{k}_{2}\right\Vert >\varepsilon\right\} $
for some $\varepsilon>0$. For $x\in B_{\varepsilon}(0)$ fixed define
\[
\mathbf{k}:=\frac{1}{2}\mathbf{k}_{1}+\frac{1}{2}\mathbf{k}_{2}\text{ and }\mathbf{k}_{x}(\omega):=\mathbf{k}(\omega)+\begin{cases}
x & \text{if }\omega\in A,\\
0 & \text{if }\omega\not\in A.
\end{cases}
\]
By construction, $\mathbf{k}$ and $\mathbf{k}_{x}$ are measurable
selections. However, we have that $\int\mathbf{k}_{x}\mathrm{d}P=\int\mathbf{k}\mathrm{d}P+x\cdot P(A)$.
As $x\in B_{\varepsilon}(0)$ was chosen arbitrarily it follows that
$\int\mathbf{k}\mathrm{d}P+B_{\varepsilon\cdot P(A)}(0)\in\E Y$,
i.e., $\frac{1}{2}k_{1}+\frac{1}{2}k_{2}$ is in the interior of $\E Y$,
which is the assertion.
\end{proof}
\begin{thm}
\label{thm:Extreme}Let $e$ be an extreme point of $\E Y$. Then
there is a unique measurable selection $\mathbf{e}(\cdot)$ with $e=\int\mathbf{e}\mathrm{d}P$
and further, $\mathbf{e}(\omega)$ is an extreme point of $Y(\omega)$
$P$\nobreakdash-almost everywhere.
\end{thm}

\begin{proof}
We notice first that $k_{1}=k_{2}$ in Definition~\ref{def:Strict}~\ref{enu:extreme}
is equivalent to $f_{i}(k_{1})=f_{i}(k_{2})$, where $f_{i}$, $i=1,\dots,d$
are linearly independent functionals.

As $e$ is contained in the boundary, $e\in\partial\E Y$, the Hahn\textendash Banach
theorem provides a linear functional $f_{d}$ so that $f_{d}(e)\ge f_{d}(y)$
for all $y\in\E Y$. Then, by Proposition~\ref{thm:4}, we have that
\begin{equation}
e\in\partial s_{\E Y}(f_{d})=\E\partial s_{Y}(f_{d}).\label{eq:7-1}
\end{equation}
It follows from Lemma~\ref{thm:SupArgMax} that $Y_{d-1}(\omega):=\partial s_{Y(\omega)}(f_{d})=\argmax_{y\in Y(\omega)}f_{d}(y)$
is contained in an affine subspace of co-dimension~$1$ parallel
to $\left\{ f_{d}(\cdot)=0\right\} $ for each $\omega\in\Omega$,
as $f_{d}$ is linear and $Y_{d-1}\subset\left\{ f_{d}(\cdot)=const\right\} $
for some constant.

From~(\ref{eq:7-1}) we deduce that $e\in\E Y_{d-1}$ and $e$, by
linearity, is an extreme point of the set $\argmax_{y\in\E Y}f_{d}(y)$,
which is contained in an affine subspace, which is of co-dimension~$1$
and parallel to $\left\{ f_{d}(\cdot)=0\right\} $ as well. 

We argue now by induction on the dimension. To this end set $Y_{d}:=Y$
and assume that $Y_{i}$ is contained in an affine subspace of co-dimension
$d-i$ so that $f_{j}(y)=f_{j}(y^{\prime})$ for all $y$, $y^{\prime}\in Y_{i}$
and $j>i$. Then we may repeat the previous argument and find a linear
functional $f_{i}$ separating $e$ and $\bigcap_{j>i}\argmax_{y\in\E Y}f_{j}(y)$.
The linear functions $f_{i}$ may be chosen linearly independent from
$f_{j}$, $j>i$, as $\bigcap_{j>i}\argmax_{y\in\E Y}f_{j}(y)$ is
contained in an affine subspace of co-dimension $d-i$.

Define recursively the random sets 
\[
Y_{i-1}(\omega):=\partial s_{Y_{i}(\omega)}(f_{i})\subset Y_{i}(\omega),
\]
which are contained in an affine hyperplane of co-dimension $d-(i-1)$
parallel to $\left\{ f_{j}(\cdot)=0\colon j=i,\dots,d\right\} $.

It follows that $Y_{1}(\omega)$ is an interval and the random variable
$\omega\mapsto Y_{1}(\omega)$, by construction, is measurable. Hence
$e\in\E Y_{1}=\int\mathbf{e}\mathrm{d}P$ and $\mathbf{e}\in Y_{1}$
is unique, as $e$ is an extreme point in the \emph{interval} 
\[
\bigcap_{i=1}^{d}\argmax_{y\in\E Y}f_{i}(y).
\]
Clearly, $\mathbf{e}\in Y_{1}\subset Y_{2}$ and $\mathbf{e}$ is
unique in $Y_{2}$ as well, as otherwise in conflict with maximality
with respect to $f_{2}$. This argument can be repeated (in a backwards
recursive way) to see that $\mathbf{e}\in Y$ is unique almost everywhere.
\end{proof}

\section{\label{sec:LLN}The law of large numbers and the central limit theorem}

To study the law of large numbers we consider a sequence of independent,
set-valued random variables $Y_{i}$ with identical distribution (i.i.d.).
We are interested in which sense the sample means $\frac{1}{N}\sum_{i=1}^{N}Y_{i}(\omega)$
converge to the expected value set $\E Y$. 

We start with general observations regarding the sample mean process. 

\subsection{\label{subsec:Convexification}Convexification}

As was discussed in Section~\ref{sec:Expectation}, the expected
value $\E Y$ is convex in many, but not all situations. However,
the sample means $\frac{1}{N}\sum_{i=1}^{N}Y_{i}$ always converge
to a convex set in the sense of the next lemma.
\begin{lem}[\citet{Artstein1985}]
Let $(K_{i})_{i=1}^{\infty}$ be a sequence of compact sets in a
Banach space $X$ such that 
\[
\frac{1}{N}\left(\conv K_{1}+\conv K_{2}+\dots+\conv K_{N}\right)\xrightarrow[N\to\infty]{}K_{0}
\]
 in Pompeiu\textendash Hausdorff distance for some convex and compact
set $K_{0}$. Then 
\[
\frac{1}{N}\left(K_{1}+K_{2}+\dots+K_{N}\right)\xrightarrow[N\to\infty]{}K_{0}.
\]
\end{lem}

\begin{rem}[Shapley-Folkman-Starr]
 The theorem by Shapley-Folkman-Starr (cf.\ \citet{Arrow1971} and
also \citet{Artstein1985}) provides an explicit bound for comparing
sums of compacts sets in the space $X=\mathbb{R}^{d}$ with finite
dimension $d$. The theorem states that 
\[
\mathbb{H}\left(\frac{1}{N}(K_{1}+K_{2}+\dots+K_{N}),\,\frac{1}{N}\conv(K_{1}+K_{2}+\dots+K_{N})\right)\le\frac{\sqrt{d}}{N}\max_{i=1,\dots,N}\bigl\Vert K_{i}\bigr\Vert,
\]
where $\bigl\Vert K\bigr\Vert=\max_{k\in K}\bigl\Vert k\bigr\Vert$
(cf.\ also \citet[Section 3.1.1]{Molchanov} and Figure~\ref{fig:ALineB}
again for illustration). 
\end{rem}

It is thus clear that the sample average $\frac{1}{N}\sum_{i=1}^{N}Y_{i}$
has the same limiting behavior as $\frac{1}{N}\sum_{i=1}^{N}\conv Y_{i}$
\textemdash the sample averages thus converge to a convex set, particularly
in the finite dimensional space~$\mathbb{R}^{d}$. For this we shall
specify further and assume the outcomes $Y_{i}(\omega)$ convex and
compact in what follows such that no separate discussion of the discrete
setting~(\ref{eq:discrete}) is necessary. 

\subsection{The set-valued law of large numbers}

By the \citeauthor{Artstein1975} Theorem \citep[p.~880]{Artstein1975},
the i.i.d.\ sample means $\overline{Y}_{N}:=\frac{1}{N}\sum_{i=1}^{N}Y_{i}$
with $\E\left\Vert Y_{i}\right\Vert <\infty$ converge indeed to the
expected value $\E Y$, i.e., 
\begin{equation}
\mathbb{H}\left(\frac{1}{N}\sum_{i=1}^{N}Y_{i},\,\E Y\right)\xrightarrow[N\to\infty]{}0\quad\text{with probability }1.\label{eq:5-2}
\end{equation}
In view of the representation of the Pompeiu\textendash Hausdorff
distance derived in~(\ref{eq:Hausdorff}) this implies in particular
that 
\[
\frac{1}{N}\sum_{i=1}^{N}s_{Y_{i}}\left(x^{*}\right)\xrightarrow[N\to\infty]{}s_{\E Y}\left(x^{*}\right)\quad\text{with probability }1
\]
 for every $x^{*}\in X^{*}$.

Eq.~(\ref{eq:5-2}) is referred to as the set-valued law of large
numbers. Several extensions are know\textcolor{blue}{n} to this fundamental
theorem, we refer the reader to \citet{XuShapiro} for a uniform law
of large numbers.

\section{\label{sec:CLT}The set-valued central limit theorem}

The CLT theorem is available in the Banach space $C\left(S^{d-1}\right)$
(cf.\ \citet{AraujoGine,Proske2003}), that is, there is a centered
Gaussian random variable $G$ on $C\left(S^{d-1}\right)$ such that
\[
\frac{1}{\sqrt{N}}\left(\sum_{i=1}^{N}s_{Y_{i}}-s_{\E Y_{i}}\right)\xrightarrow{\mathcal{D}}G,
\]
where $\xrightarrow{\mathcal{D}}$ indicates convergence in distribution,
i.e.,  $\E f(X_{n})\to\E f(X)$ for every $\mathbb{R}$\nobreakdash-valued
function $f$ which is bounded and continuous.  In full generality:
\begin{thm}[{\citet[Theorem~3]{Weil1982}}]
\label{thm:CLT}Let $\left(Y_{i}\right)_{i=1}$ and $Y$ be i.i.d.\ random
sets with $\E\left\Vert Y\right\Vert ^{2}<\infty$. Then  
\[
\sqrt{N}\cdot\mathbb{H}\left(\frac{1}{N}\sum_{i=1}^{N}Y_{i},\,\E\conv Y\right)\xrightarrow{\mathcal{D}}\left\Vert G\right\Vert _{\infty},
\]
where $G$ is a centered Gaussian $C\left(S^{d-1}\right)$\nobreakdash-valued
random variable.
\end{thm}

\begin{proof}
Cf.\ \citet[Theorem~8]{Weil1982} or \citet{Casey2000}. The proof
is based on computing the metric entropy of $S^{d-1}$ and the respective
bracketing numbers, it reduces the particular situation here to the
general situation described in \citet{AraujoGine}. Elements of the
general theory and proofs can be found in \citet{Vaart1996}.
\end{proof}
The Gaussian measure $G$ in Theorem~\ref{thm:CLT} is provided by
the isometry of convex and compact sets with their respective support
function. Moreover $\mathbb{H}\left(\cdot,\cdot\right)\in\mathbb{R}_{\ge0}$
always is just a positive number (as is $\left\Vert G\right\Vert _{\infty}$),
the statement just considers the $\mathbb{R}_{\ge0}$\nobreakdash-valued
random process $\sqrt{N}\cdot\mathbb{H}\left(\frac{1}{N}\sum_{i=1}^{N}Y_{i},\,\E\conv Y\right)$
and does not reveal anything of the local convergence properties of
the sample mean to the expected value. 

\bigskip{}

In view of the latter statements, the preceding discussion and~(\ref{eq:HausdorffBoundary}),
the interesting properties are to be expected on the boundary $\partial\E Y$.
In what follows we shall distinguish and consider three particular
situations on the boundary of $\E Y$, which can be considered to
be extremal situations. We discuss the CLT for \emph{exposed points},
for \emph{tangent planes} and \emph{facets} of $\E Y$ in the following
subsections separately. 

\subsection{The CLT for exposed points}

The following theorem ensures that for any exposed point $k\in\exp\E Y$
there is a particular selection $y_{N}\in\overline{Y}_{N}:=\frac{1}{N}\sum_{i=1}^{N}Y_{i}$
from the sample means, such that the process $\sqrt{N}\left(y_{N}-k\right)$
converges to a Gaussian random variable. 
\begin{thm}
\label{thm:15}Suppose that the envelope function $h$ (cf.~(\ref{eq:Envelope}))
satisfies $h\in L^{2}$. 

Let $k\in\exp\E Y$ be exposed by the functional $f\in X^{*}$ and
$y_{N}\in\overline{Y}_{N}$ be exposed by the same $f$. Then there
is a unique measurable selection $\mathbf{k}(\cdot)\in Y(\cdot)$
such that 
\[
\sqrt{N}\left(y_{N}-k\right)\xrightarrow{\mathcal{D}}\mathcal{N}_{d}\left(0,\Sigma\right)\quad\text{as }N\to\infty,
\]
where $k=\E\mathbf{k}$ and $\Sigma$ is the covariance matrix 
\[
\Sigma:=\E\left(\mathbf{k}-k\right)\left(\mathbf{k}-k\right)^{\top}.
\]
\end{thm}

\begin{proof}
There is a measurable selection $\mathbf{k}$ such that $k=\int\mathbf{k}\mathrm{d}P$.
By Theorem \ref{thm:6} the selection $\mathbf{k}$, as $k$ is exposed,
is unique and $\mathbf{k}(\omega)\in\partial s_{Y(\omega)}\left(f\right)\subseteq Y(\omega)$.
$\mathbf{k}$ is a random variable with expectation $k$, and as $\left\Vert \mathbf{k}\right\Vert \le\left\Vert Y\right\Vert \le h\in L^{2}$
the covariance matrix 
\[
\Sigma:=\var\,\mathbf{k}=\E\left(\mathbf{k}-k\right)\left(\mathbf{k}-k\right)^{\top}
\]
 exists. 

It follows from Lemma~\ref{Lem:argmax} and Theorem~\ref{thm:6}
that $\mathbf{k}(\omega)$ and $\mathbf{k}_{i}(\omega)$, where 
\begin{equation}
\mathbf{k}_{i}(\omega)\in\partial s_{Y_{i}(\omega)}\left(f\right),\label{eq:ki}
\end{equation}
 have the same distribution for all $i$. Hence $y_{N}\in\partial s_{\bar{Y}_{N}}\left(f\right)\subseteq\overline{Y}_{N}$
and $\overline{\mathbf{k}}_{N}:=\frac{1}{N}\sum_{i=1}^{N}\mathbf{k}_{i}$
have the same distribution as well.

By the central limit theorem (cf.\ \citet{vdVaart}) thus 
\begin{equation}
\frac{1}{\sqrt{N}}\sum_{i=1}^{N}\left(\mathbf{k}_{i}-k\right)\xrightarrow{\mathcal{D}}\mathcal{N}_{d}\left(0,\Sigma\right)\quad(N\to\infty),\label{eq:CLT1-1}
\end{equation}
which is in turn 
\[
\sqrt{N}\left(y_{N}-k\right)\xrightarrow{\mathcal{D}}\mathcal{N}_{d}\left(0,\Sigma\right)
\]
because 
\[
\E\mathbf{k}_{i}=k\text{ and }\var\,\mathbf{k}_{i}=\E\left(\mathbf{k}_{i}-k\right)\left(\mathbf{k}_{i}-k\right)^{\top}=\Sigma.
\]
\end{proof}
In the proof of Theorem~\ref{thm:15} it is essential to find a measurable
selection $\mathbf{k}_{i}\in Y_{i}$ having the same distribution
as $\mathbf{k}$, which is possible by means of (\ref{eq:ki}). Similar
choices are possible in some other situations, for example again in
the binomial setting (\ref{eq:discrete}) as in Section~\ref{sec:Expectation}: 
\begin{cor}
Let $Y_{i}$ be as in Section~\ref{sec:Expectation} with the additional
assumption that $K_{j}\neq K_{j^{\prime}}$ ($j\neq j^{\prime}$).
Then, for any selection $\mathbf{k}\in Y$ with $k=\E\mathbf{k}\in\E Y$
there are selections $\mathbf{k}_{i}\in Y_{i}$ with the same distribution
as $\mathbf{k}$ such that 
\begin{equation}
\sqrt{N}\sum_{i=1}^{N}\left(\mathbf{k}_{i}-k\right)\xrightarrow{\mathcal{D}}\mathcal{N}_{d}\left(0,\var\,\mathbf{k}\right).\label{eq:conv}
\end{equation}
\end{cor}

\begin{proof}
Let 
\[
k_{j}:=\mathbf{k}(\omega)\in K_{j}\mbox{ if }Y(\omega)=K_{j}\quad(j=1,2,\dots J)
\]
and define
\[
\mathbf{k}_{i}(\omega):=k_{j}\quad\mbox{if }Y_{i}(\omega)=K_{j}\quad(i=1,2,3\dots),
\]
such that $\mathbf{k}_{i}\in\left\{ k_{1},\dots k_{J}\right\} =\text{range}(\mathbf{k})$
and 
\begin{equation}
\mathbf{k}_{i}(\omega)=\mathbf{k}(\omega)\mbox{ if }Y_{i}(\omega)=Y(\omega)\quad(i=1,2,\dots).\label{eq:ki2}
\end{equation}
Then the random variables have the same distribution, as $p_{j}=P(Y=K_{j})=P(Y_{i}=K_{j})=P(\mathbf{k}_{i}=k_{j})=P(\mathbf{k}=k_{j})$.
It follows that $\E\mathbf{k}=\E\mathbf{k}_{i}=\sum_{j}p_{j}k_{j}=k$
and $\var\,\mathbf{k}_{i}=\var\,\mathbf{k}=\E\left(\mathbf{k}-k\right)\left(\mathbf{k}-k\right)^{\top}=\sum_{j}p_{j}\left(k_{j}-k\right)\left(k_{j}-k\right)^{\top},$
from which the rest is immediate.
\end{proof}
\begin{rem}
A point $k\in\E Y$ may have various selections $\mathbf{k}$ with
$k=\E\mathbf{k}\in\E Y$, there are hence various selections $\mathbf{k}_{i}\in Y_{i}$
with possibly different convergence behavior~(\ref{eq:conv}). However,
if $k$ is an exposed point, then the selection $\mathbf{k}$ is unique
by Theorem~\pageref{thm:15} and the selections~(\ref{eq:ki}) and~(\ref{eq:ki2})
coincide.
\end{rem}

\subsection{The CLT along tangent planes}

Any compact and convex $K$ can be given as $K=\bigcap_{x^{*}\in X^{*}}\left\{ x^{*}(\cdot)\le\max_{k\in K}x^{*}(k)\right\} $
and $\left\{ x^{*}(\cdot)=\max_{k\in K}x^{*}(k)\right\} $ is a tangent
plane of co-dimension $1$. While the previous subsection addresses
exposed points for which $K\cap\left\{ x^{*}=\max x^{*}(K)\right\} =\partial s_{K}\left(x^{*}\right)$
is a singleton, we continue in this subsection with the situation
that $\partial s_{K}\left(x^{*}\right)$ is not necessarily a singleton. 

\medskip{}

The law of large numbers does not only hold for the sequence $Y_{i}$,
it applies for subdifferentials as well.
\begin{prop}
\label{cor:19}Let $Y$ and $Y_{i}$ be independent and identically
distributed, compact and convex valued random sets with $L^{2}$\nobreakdash-envelope
(cf.~(\ref{eq:Envelope})). Then 
\begin{equation}
\mathbb{H}\left(\frac{1}{N}\sum_{i=1}^{N}\partial s_{Y_{i}}\left(x^{*}\right),\ \partial_{\E Y}(x^{*})\right)\xrightarrow{N\to\infty}0\qquad\text{with probability }1\label{eq:64}
\end{equation}
for any $x^{*}$, and moreover 
\begin{equation}
\frac{1}{\sqrt{N}}\sum_{i=1}^{N}\bigl(\max_{y\in Y_{i}}x^{*}\left(y\right)-\max_{y\in\E Y}x^{*}\left(y\right)\bigr)\xrightarrow{\mathcal{D}}\mathcal{N}\left(0,\,\sigma^{2}\left(x^{*}\right)\right),\label{eq:maxxs}
\end{equation}
where $\sigma^{2}\left(x^{*}\right)=\var\,\max_{y\in Y}x^{*}(y)$.
\end{prop}

\begin{proof}
Notice first that $\partial s_{Y}\left(x^{*}\right)\subseteq Y$,
and the law of large numbers applies to the sequence $\partial s_{Y_{i}}\left(x^{*}\right)\subseteq Y_{i}$
as well: that is to say $\mathbb{H}\left(\frac{1}{N}\sum_{i=1}^{N}\partial s_{Y_{i}}\left(x^{*}\right),\ \E\partial_{Y}\left(x^{*}\right)\right)\xrightarrow{N\to\infty}0$
with probability $1$, and the identity $\E\partial_{Y}\left(x^{*}\right)=\partial_{\E Y}\left(x^{*}\right)$
already was established in Theorem~\ref{thm:4}. 

As for~(\ref{eq:maxxs}) note that 
\[
\max_{y\in Y_{i}}x^{*}(y)=x^{*}\left(\argmax_{Y_{i}}x^{*}\right)=x^{*}\left(\partial s_{Y_{i}}\left(x^{*}\right)\right)
\]
is an $\mathbb{R}$\nobreakdash-valued random variable and square
integrable, because $\left|\max_{y\in Y_{i}(\omega)}x^{*}(y)\right|\le\left\Vert x^{*}\right\Vert _{*}\cdot h(\omega)$.
Further, $\E\max_{y\in Y}x^{*}(y)=\E x^{*}\left(\partial s_{Y}\left(x^{*}\right)\right)=x^{*}\left(\E\partial s_{Y}\left(x^{*}\right)\right)=x^{*}\left(\partial s_{\E Y}\left(x^{*}\right)\right)=\max_{y\in\E Y}x^{*}(y)$.
The asymptotic distribution~(\ref{eq:maxxs}) thus follows from the
classical Central Limit Theorem for $\mathbb{R}$\nobreakdash-valued
random variables.
\end{proof}
\begin{rem}
Proposition~\ref{cor:19} reduces the original problem into two distinct,
orthogonal problems, as~(\ref{eq:64}) describes the behavior of
parallel sets, all of co-dimension one, whereas~(\ref{eq:maxxs})
is their orthogonal component.
\end{rem}

\subsection{The CLT for facets}

A functional $f\in X^{*}$ induces the particular selection~(\ref{eq:ki})
above by exposing a single point of the boundary of $\E Y$. With
this selection it was possible to describe convergence of corresponding
exposed points of the sample means. 

In what follows we take a kind of dual approach and fix a vector $x\in X$
first. Then there are nearest points to a compact, convex set $K$,
which we denote by 
\begin{equation}
k_{x}(K)\in\argmin\left\{ \left\Vert x-k^{\prime}\right\Vert :\,k^{\prime}\in K\right\} .\label{eq:5-8}
\end{equation}
In order to have $k_{x}(\cdot)$ uniquely defined we shall assume
that the unit ball of the norm $B_{\left\Vert \cdot\right\Vert }$
is strictly convex (cf.\ Definition~\ref{def:Strict}~\ref{enu:strict}
and Figure~\ref{fig:FacetHB-1}). We consider the random variable
$\mathbf{k}(\omega):=k_{x}\left(K(\omega)\right)\in K(\omega),$ which
is a particular selection, whose convergence behavior is being elaborated
in what follows.
\begin{center}
\begin{figure}
\centering{}\subfloat[Facet]{\includegraphics[width=0.5\textwidth]{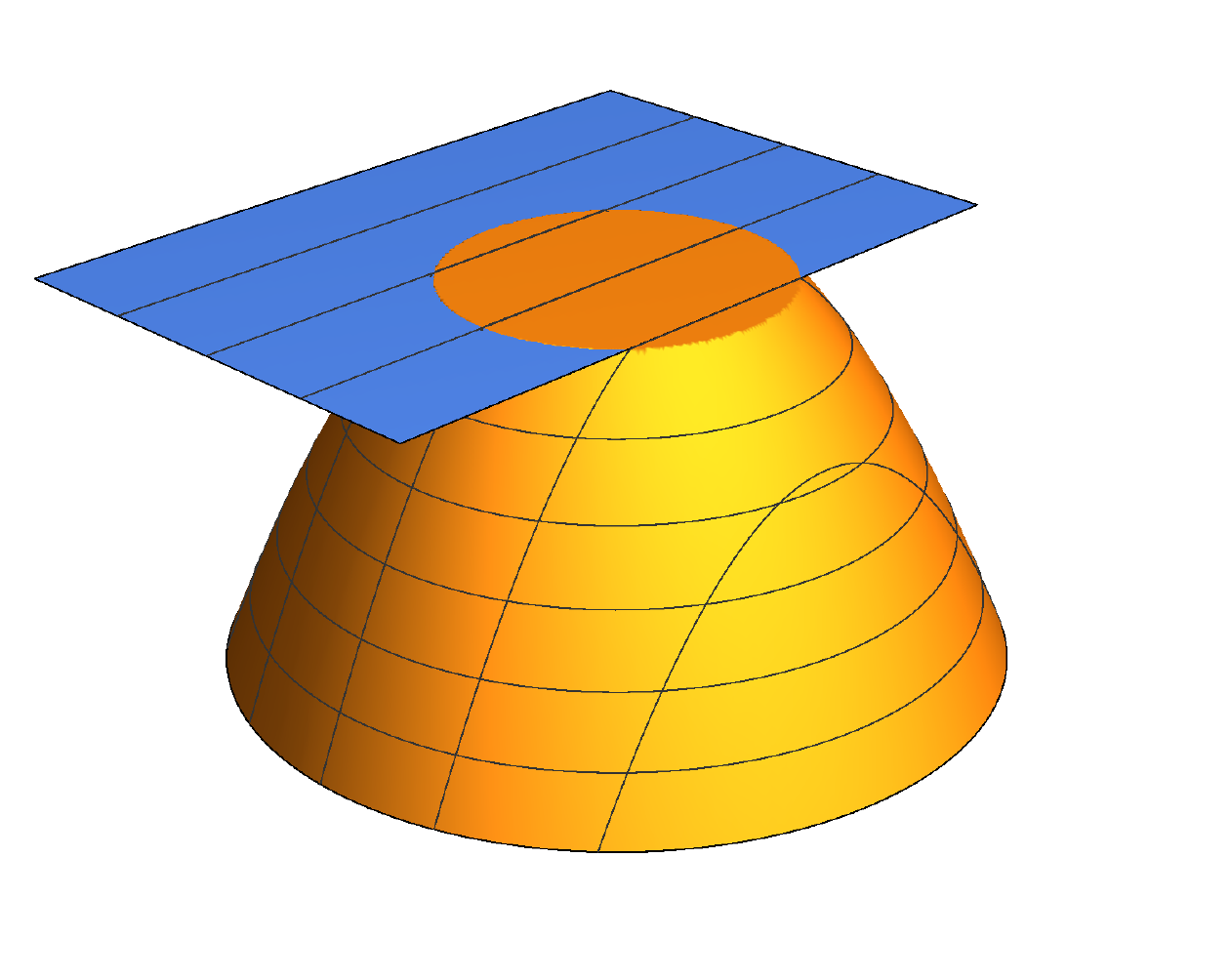}}\hfill{}\subfloat[Facet $\HB_{k-x}$ at $k\in\E X$: For all $u\in U$, $u+\frac{\HB_{k-x}(k-u)}{\left\Vert k-x\right\Vert }(k-x)\in\E X$.\label{fig:FacetHB-1}]{\begin{tikzpicture}[join=round, scale= 0.8]

\filldraw[fill=blue!20,fill opacity=0.8](-.6,-2.8){[rounded corners=20]--(3.1,-1.4)}--(2,3)--(-.3,3.6);
\draw[blue, ultra thick](3.26,-2.0)--(1.76,4);	
\filldraw[black] (2.3,-1.7) circle (2pt) node[anchor=north] {$e$};
\node  at (.2,2.7) {$\mathbb E X$};

\draw[arrows=->,very thick](5.1,-.1)--(2.56,1.02); 
\filldraw[black] (5.1,-.1) circle (2pt) node[anchor=south] {$x$};

\draw[thick, gray](2.5,1) circle (.8);	
\draw[arrows=->, very thick](2.5,1)--(.5,0.5);
\filldraw[black] (2.5,1.0) circle (2pt) node[anchor=south west] {$k$};

\draw[rotate around={30:(2.2,9.4)}] ellipse (2 and 4);
\node at (1.2,0.2) {$\HB_{k-x}$};

\node at (3,2) {$U$};
\end{tikzpicture}}\caption{Facets\label{fig:Facets}}
\end{figure}
\par\end{center}
\begin{defn}[Facet]
\label{def:Facet}A (continuous) linear functional $f\neq0$ is a
\emph{facet at $k\in K$ }if there is a direction $d$ (associated
with $f$) and a neighborhood $U(k)$ such that $x-d\cdot f(x-k)\in\argmax_{k^{\prime}\in K}f(k^{\prime})$
for all $x\in U(k)$. Further, we shall say that \emph{$k\in K$ is
contained in a facet }if there exists a facet at $k\in K$.
\end{defn}

We collect the following important features of facets, as they will
be of interest in what follows (cf.\ Figure~\ref{fig:Facets} for
a simple, helpful illustration). 
\begin{rem}[Important properties of facets]
Let $f$ be a facet and $d$ a direction associated with the facet
 according to Definition~\ref{def:Facet}. 
\begin{enumerate}
\item \label{enu:d}The direction $d$ associated with the facet $f$ always
satisfies $f(d)=1$: to see this note first that necessarily $k\in\argmax_{K}f$,
as $k\in U(k)$. For $x\in U(k)$ fixed thus, $f(x)-f(d)\cdot f(x-k)=f(k)$,
hence $f(x-k)=f(d)\cdot f(x-k)$ for all $x\in U(k)$, which can hold
true only if $f(d)=1$.
\item Associated with a facet and a direction $d$ are the projection operators
$P^{\bot}:=\frac{d\otimes f}{f(d)}$ and $P:=1-P^{\bot}$, where $P^{\bot}(x)=\frac{d}{f(d)}f(x)$.
Indeed, it follows from~\ref{enu:d} that $P^{\bot}\circ P^{\bot}=P^{\bot}$,
and thus $P\circ P=P$. In the context of facets of the expected value
set below we consider the shifted projective map $x\mapsto k+P^{\bot}(x-k)$. 
\item A facet\textemdash up to a constant\textemdash is unique. To accept
this let $f$ be a facet, that is $x-d\cdot f(x-k)\in\argmax_{k^{\prime}\in K}f(k^{\prime})\subseteq K$.
For another facet $g$ hence $g\left(x-d\cdot f(x-k)\right)\le g(k)$,
that is 
\begin{equation}
g(x-k)\le g(d)\cdot f(x-k).\label{eq:facet}
\end{equation}
For $x\in U(k)$, $2k-x\in U(k)$ as well (at least for $x$ close
enough to $k$). Hence 
\[
g(2k-x-k)\le g(d)\cdot f(2k-x-k),
\]
or $g(x-k)\ge g(d)\cdot f(x-k)$, which, together with (\ref{eq:facet}),
implies that $g(x-k)=g(d)\cdot f(x-k)$, and hence $g(\cdot)=g(d)\cdot f(\cdot)$.
\item We have that $x-d\cdot f(x-k)\in U_{\varepsilon}(k)\subseteq\argmax_{K}f\subseteq K$
for some small $\varepsilon>0$. Indeed, for $\varepsilon>0$ small
enough and $x\in U_{\varepsilon/\left(1+\left\Vert d\right\Vert \left\Vert f\right\Vert _{*}\right)}(k)\subseteq U(k)$
we have that $\left\Vert x-d\cdot f(x-k)-k\right\Vert \le\left(1+\left\Vert d\right\Vert \left\Vert f\right\Vert _{*}\right)\cdot\left\Vert x-k\right\Vert <\varepsilon$,
and hence $x-d\cdot f(x-k)\in U_{\varepsilon}(k)\subseteq\argmax_{K}f\subseteq K$.
\item The direction $d^{\prime}$ of the facet can be chosen arbitrarily,
as long as $f\left(d^{\prime}\right)=1$.  Indeed, recall that $x-d^{\prime}\cdot f(x-k)\in U_{\varepsilon}(k)$.
As $f\left(x-d^{\prime}\cdot f(x-k)\right)=f(k)$ we find further
that $x-d^{\prime}\cdot f(x-k)\in\argmax_{K}f$, which is the assertion
for the alternative direction $d^{\prime}$ whenever $x\in U_{\varepsilon/\left(1+\left\Vert d^{\prime}\right\Vert \cdot\left\Vert f\right\Vert _{*}\right)}(k)$
as above.
\end{enumerate}
\end{rem}

We demonstrate next that the expected value set $\E Y$ inherits all
facets from the sample sets $Y_{i}$.
\begin{prop}
\label{thm:exp+facet}Let $K$ be a convex and compact set with facet
$f$ and $Y$ be a set-valued random variable with $P(Y=K)>0$. 
\begin{enumerate}
\item \label{enu:facet-1}Then $\E Y$ has a facet as well; more precisely,
$f$ is a facet of $\E Y$ at each point in the relative interior
of $\E\partial s_{Y}(f)$;
\item \label{enu:facet}Let $Y$ and $\left(Y_{i}\right)$ be i.i.d.\ random
sets as in the discrete setting~(\ref{eq:discrete}). Then $f$ is
a facet of $\overline{Y}_{N}$ with probability $1-(1-p)^{N}$, where
$p:=P(Y=K)$.
\end{enumerate}
\end{prop}

\begin{proof}
Let $K$ have a facet $f$ at some $k\in K$. Then 
\[
\argmax_{y\in K}f(y)\supseteq\left\{ x-d\cdot f(x-k)\colon x\in B_{r}(k)\right\} 
\]
for some $r>0$ and $d$ with $f(d)=1$. Choose again the selection
$\mathbf{k}(\omega)=\partial s_{Y(\omega)}\left(f\right)$. By~(\ref{Lem:argmax})
thus\footnote{The conditional expectation is understood in the naïve sense based
on conditional probabilities here: note that the sets $\left\{ Y=K\right\} $
and $\left\{ Y\not=K\right\} $ have strictly positive probability.} 
\begin{align*}
\argmax_{y\in\E Y} & f(y)=\E\argmax_{y\in Y}f(y)\\
 & =P(Y\neq K)\cdot\E\left[\left.\argmax_{y\in Y}f(y)\right|\,Y\neq K\right]+P(Y=K)\cdot\E\left[\left.\argmax_{y\in Y}f(y)\right|\,Y=K\right]\\
 & =P(Y\neq K)\cdot\E\left[\left.\argmax_{y\in Y}f(y)\right|Y\neq K\right]+P(Y=K)\cdot\argmax_{y\in K}f(y)\\
 & \supseteq P(Y\neq K)\cdot\E\left[\mathbf{k}|Y\neq K\right]+P(Y=K)\cdot\left\{ x-d\cdot f(x-k)\colon x\in B_{r}(k)\right\} \\
 & =P(Y\neq K)\cdot\E\mathbf{k}+P(Y=K)\cdot\left\{ k+x-d\cdot f(x)\colon x\in B_{r}(0)\right\} \\
 & =P(Y\neq K)\cdot\E\,k+P(Y=K)\cdot k+\left\{ x-d\cdot f(x)\colon x\in B_{r\cdot P(Y=K)}(0)\right\} ,
\end{align*}
hence $f$ is a facet of $\E Y$ at every $k^{\prime}\in P(Y\neq K)\cdot\E\mathbf{k}+P(Y=K)\cdot k\subset\E\mathbf{k}$.
(Recall that $P(Y=K)>0$, the conditional expectation in the previous
display thus does not cause difficulties).

As for~\ref{enu:facet} note that there is $i^{*}$ so that $K=K_{i^{*}}$.
By~\ref{enu:facet-1}, $\overline{Y}_{N}$ has the facet $f$ as
soon as $Y_{i}=K$, which happens with the probability $P(Y_{i}=K)\ge1-(1-p)^{N}$
at the $N$\nobreakdash-th draw. 
\end{proof}
\begin{rem}[The converse is false]
Figure~\Floref{NoFacet} provides an example of two sets $A$ and
$B$ without facets, although their average $\nicefrac{1}{2}(A+B)$
has a facet. Hence if $\E Y$ has a facet, then this is not necessarily
the case for $Y$, not even for discrete random variables $Y$. 
\begin{figure}
\begin{centering}
\includegraphics[width=0.3\textwidth]{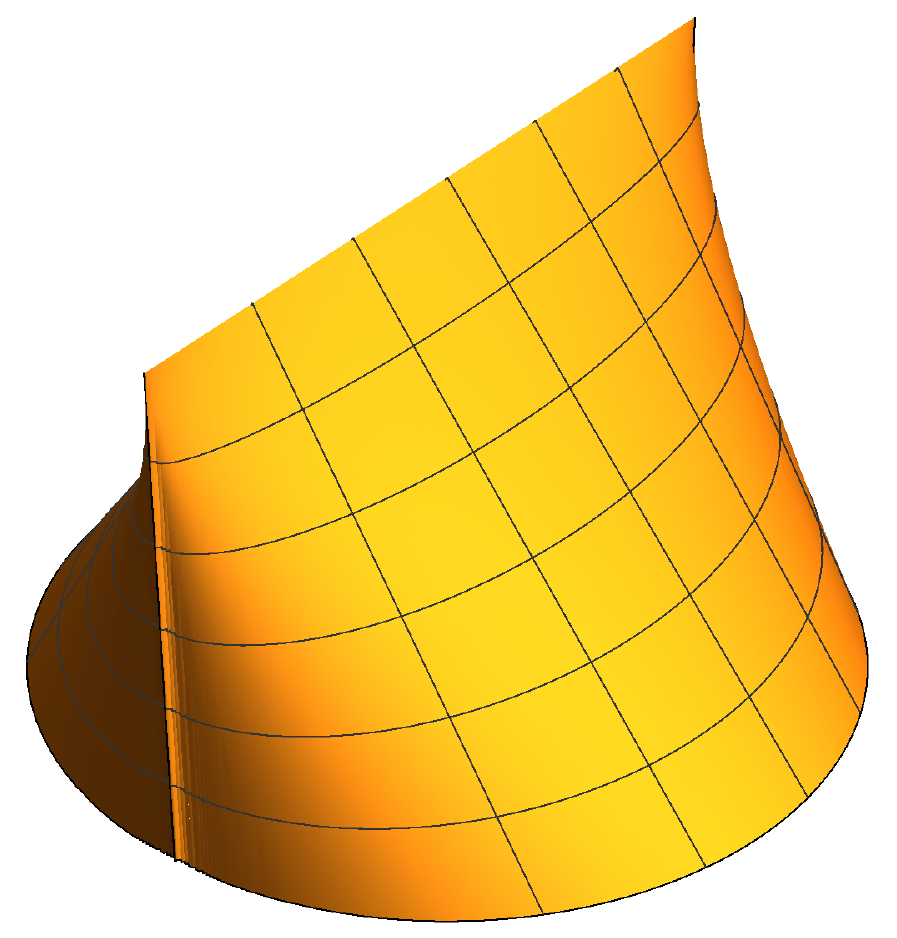}\includegraphics[width=0.3\textwidth]{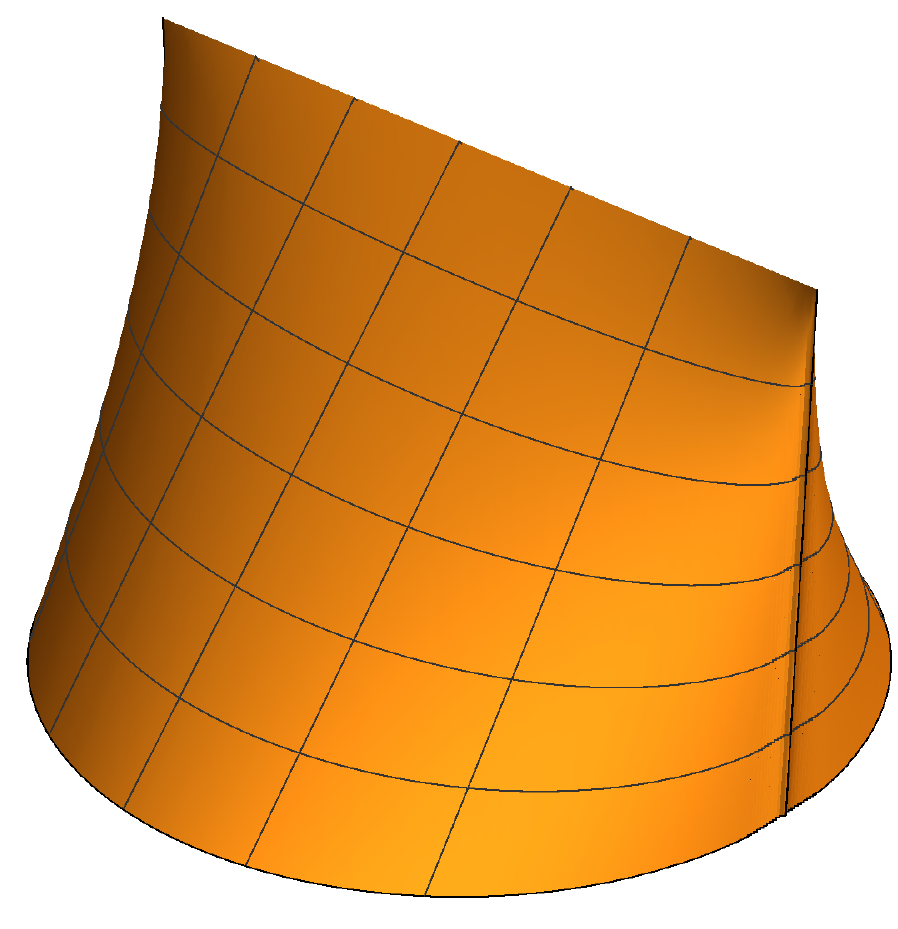}
\par\end{centering}
\caption{\label{Flo:NoFacet}$\frac{1}{2}(A+B)$ has a facet at its top, although
$A$ and $B$ have no facet (cf.\ Figure~\ref{fig:ALineB}) (the
depicted solid's equation is $x^{2}z^{2}+y^{2}\le z^{2}$).}
\end{figure}
\end{rem}

To describe the convergence of set-valued sample means close to a
facet of $\E Y$ it will be convenient to have an outer normal available.
The facet normal is given by the derivative of the norm (cf.~Figure~\ref{fig:FacetHB-1}
for an illustration with an elliptic unit ball, and \citet{Vitale2000}
for facet normals). 
\begin{defn}[Derivative of the Norm]
We shall denote an element of the derivative of the norm $x\in\mathbb{R}^{d}$
by $\HB_{x}$,
\[
\HB_{x}\in\partial s_{B^{*}}(x)\subseteq\mathbb{R}^{d};
\]
here, $B^{*}:=\left\{ x\in\mathbb{R}^{d}:\,\left\Vert x\right\Vert _{*}\le1\right\} $
is the unit ball in the dual space.
\end{defn}

\begin{rem}
By~(\ref{eq:supportFctn}) and~(\ref{eq:Argmax}) it holds that
\[
\HB_{x}(x)=\left\Vert x\right\Vert \text{ and }\left|\HB_{x}(h)\right|\le\left\Vert h\right\Vert 
\]
for all $h\in\mathbb{R}^{d}$ (that is to say the norm in the dual
is one, $\left\Vert \HB_{x}\right\Vert _{*}=1$, where the norm is
the Lipschitz constant $\left\Vert \lambda\right\Vert _{*}:=\sup_{h\neq0}\frac{\left|\lambda(h)\right|}{\left\Vert h\right\Vert }=L(\lambda)$).
\end{rem}

\begin{thm}
Given $x\notin\overline{\E Y}$, suppose that $k\in\argmin\left\{ \left\Vert x-y\right\Vert :\,y\in\E Y\right\} $,
the closest point to $x$, is contained in a facet $f$ (cf.~(\ref{eq:5-8})).
Then the facet satisfies $f(\cdot)=-\alpha\cdot\HB_{k-x}(\cdot)$
for some $\alpha>0$.
\end{thm}

\begin{proof}
Given $x$, choose $k$ the nearest point in $\E Y$ such that $d\left(x,\E Y\right)=\left\Vert x-k\right\Vert $.
Both, $\E Y$ and the ball $B_{\left\Vert x-k\right\Vert }(x)$ are
convex, and $k$ is a common point. Moreover $\E Y$ and the open
ball $\overset{\circ}{B}_{\left\Vert x-k\right\Vert }(x)$ do not
intersect. The Hahn\textendash Banach Theorem provides a functional
(separating plane) for both sets. As the facet is unique the separating
functional $\HB_{k-x}(\cdot)$ is the facet.
\end{proof}
\begin{thm}
\label{thm:name}Given $x\notin\overline{\E Y}$, suppose that $k_{x}\left(\E Y\right)$,
the nearest point to $x$, is contained in a facet of $\E Y$. Then
there is a neighborhood $V(x)$ such that the Pompeiu\textendash Hausdorff
distance is $\mathbb{H}\left(\left\{ v\right\} ,\,\E Y\right)=\HB_{k-x}(k-v)$
for all $v\in V$, and moreover $\mathbb{H}\left(\left\{ v\right\} ,\,\E Y\right)-\mathbb{H}\left(\left\{ x\right\} ,\,\E Y\right)=\HB_{k-x}(x-v)$.
\end{thm}

\begin{proof}
By the above theorem the facet is $-\HB_{k-x}$. Let us equip the
facet $-\HB_{k-x}$ with the direction $d:=-\frac{k-x}{\left\Vert k-x\right\Vert }$,
such that $-\HB_{k-x}(d)=1$. Being a facet, there is by definition
a neighborhood $U(k)$ such that $u-d\cdot\HB_{k-x}(u-k)\in\argmax_{\E X}\left(-\HB_{k-x}\right)$
for all $u\in U(k)$. Define $V(x):=U(k)-(k-x)$. Then $v+\frac{\HB_{k-x}(k-v)}{\left\Vert k-x\right\Vert }(k-x)\in\argmax_{\E Y}\left(-\HB_{k-x}\right)$
for every $v\in V(x)$. Hence, as $k\in\E Y$, 
\begin{align}
\mathbb{H}\left(\left\{ v\right\} ,\,\E Y\right) & =\HB_{k-x}(k-v).\label{eq:19}
\end{align}
The latter statement of the theorem follows from linearity, as $\mathbb{H}\left(\left\{ v\right\} ,\,\E Y\right)-\mathbb{H}\left(\left\{ x\right\} ,\,\E Y\right)=\HB_{k-x}(k-v)-\HB_{k-x}(k-x)=\HB_{k-x}(x-v).$
\end{proof}
With these preparations we can finally describe the distribution along
facets.
\begin{thm}
Given $x$, suppose that $k_{x}\left(\E Y\right)$, the nearest point
to $x$, is contained in a facet of $\E Y$. Then 
\begin{equation}
\sqrt{N}\left(\mathbb{H}\left(\{x\},\overline{Y}_{N}\right)-\mathbb{H}\left(\{x\},\E Y\right)\right)\xrightarrow{\mathcal{D}}\mathcal{N}\left(0,\,\HB_{k-x}\cdot\Sigma\cdot\HB_{k-x}^{\top}\right)\label{eq:HBDist-1}
\end{equation}
where $k$ and $\Sigma$ are as in Theorem~\ref{thm:15} and $\overline{Y}_{N}:=\frac{1}{N}\sum_{i=1}^{N}Y_{i}$.
\end{thm}

\begin{proof}
Note that $\mathbf{k}(\omega)=k_{x}\big(Y(\omega)\big)$ is almost
surely uniquely defined as the norm is strictly convex and $k=\E\mathbf{k}$.
We define the random quantities $\mathbf{k}_{i}:=k_{x}(Y_{i})$ and
\begin{equation}
V_{i}:=\mathbf{k}_{i}+x-k\label{eq:19-1}
\end{equation}
($\overline{\mathbf{k}}_{N}:=\frac{1}{N}\sum_{i=1}^{n}\mathbf{k}_{i}$
and $\overline{V}_{N}:=\frac{1}{N}\sum_{i=1}^{N}V_{i}$, resp.), such
that 
\begin{equation}
\frac{1}{\sqrt{N}}\sum_{i=1}^{N}\HB_{k-x}\left(x-V_{i}\right)\xrightarrow{\mathcal{D}}\mathcal{N}\left(0,\,\HB_{k-x}\cdot\Sigma\cdot\HB_{k-x}^{\top}\right),\quad n\to\infty,\label{eq:20}
\end{equation}
where $\Sigma=\E(\mathbf{k}-k)(\mathbf{k}-k)^{\top}$. Note next that
$\E\frac{1}{N}\sum_{i=1}^{N}V_{i}=x$. By Theorem~\ref{thm:name}
there is a neighborhood $V(x)$ such that 
\begin{equation}
\mathbb{H}\left(\left\{ v\right\} ,\,\E Y\right)-\mathbb{H}\left(\left\{ x\right\} ,\,\E Y\right)=\HB_{k-x}(x-v)\qquad\text{ for all }v\in V(x).\label{eq:17}
\end{equation}
Note further that 
\begin{align*}
\mathbb{H} & \left(\left\{ x\right\} ,\,\E Y\right)-\mathbb{H}\left(\left\{ x\right\} ,\,\overline{Y}_{N}\right)\\
 & =\left(\mathbb{H}\left(\left\{ x\right\} ,\,\E Y\right)-\mathbb{H}\left(\left\{ \overline{V}_{N}\right\} ,\E Y\right)\right)\\
 & \quad+\mathbb{H}\left(\left\{ \overline{V}_{N}\right\} ,\,\E Y\right)-\mathbb{H}\left(\left\{ \overline{V}_{N}\right\} ,\overline{Y}_{N}\right)+\left(\mathbb{H}\left(\left\{ \overline{V}_{N}\right\} ,\overline{Y}_{N}\right)-\mathbb{H}\left(\left\{ x\right\} ,\overline{Y}_{N}\right)\right)\\
 & =\HB_{k-x}(\overline{V}_{N}-x)+\HB_{k-x}(k-\overline{V}_{N})-\left\Vert x-k\right\Vert +\HB_{\overline{\mathbf{k}}_{N}-x}(x-\overline{V}_{N})\\
 & =\HB_{\overline{\mathbf{k}}_{N}-x}(x-\overline{V}_{N})
\end{align*}
where we have used~(\ref{eq:17}),~(\ref{eq:19}),~(\ref{eq:19-1})
and again~(\ref{eq:17}), provided that $\overline{V}_{N}\in V(x)$.
The assertion of the theorem follows from~(\ref{eq:20}) as $\HB_{\overline{\mathbf{k}}_{N}-x}\to\HB_{k-x}$
for the strictly convex norm, provided that can ensure that $\overline{V}_{N}\in V(x)$
almost surely.

As $x$ is in the interior of $V(x)$ we apply the large deviation
theory (cf.\ for example \citet{Dembo1998} or \citet[Theorem~4.1]{Norkin2012})
to obtain that 
\[
\limsup_{N\to\infty}\frac{1}{N}\ln P\left(\frac{1}{N}\sum_{i=1}^{N}V_{i}\notin V(x)\right)<0.
\]
That is, there is $q>0$ such that $P\left(\frac{1}{N}\sum_{i=1}^{N}V_{i}\notin V(x)\right)<e^{-qN}$
and thus 
\[
P\left(\frac{1}{N}\sum_{i=1}^{N}V_{i}\in V(x)\right)>1-e^{-qN}\xrightarrow{N\rightarrow\infty}1.
\]
The desired distribution~(\ref{eq:HBDist-1}) follows hence from~(\ref{eq:20}).
\end{proof}

\section{\label{sec:Summary}Summary}

We discuss convergence properties of random sets. We are particularly
interested in fluctuations of the sample means close to the boundary
of the limit set, the expected value. It turns out that special properties
of points on the boundary of the expected value set can already be
seen at the boundary of the sample means, while other properties are
inherited from the sample means to the expected value set. 

The paper addresses important boundary points of the expected value
set separately. Exposed points of the expected value set have a unique
measurable selection, and so have the sample means. Convergence thus
can be described by a usual process of points in $\mathbb{R}^{d}$.
Tangent planes display a similar behavior, we describe their convergence
by identifying the moments to describe their convergence by use of
the central limit theorem.

We finally address facets which are inherited by the expected value
set, but (perhaps surprisingly) not the other way round. 

\section{Acknowledgment}

\begin{wrapfigure}{o}{0.1\textwidth}%
\centering{}\includegraphics[clip,width=0.1\textwidth]{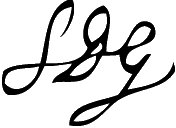}\end{wrapfigure}%
Special thanks to Prof.\ Roger J.-B.\ Wets and Prof.\ Georg Ch.\ Pflug,
who encouraged and supported investigating set-valued mappings. Both
provided useful comments on initial versions of this paper. 

We would like to thank the editor of the journal and two independent
referees for their commitment to assess the paper. Their comments
were very professional and profound and lead to a significant improvement
of the content. 

\bibliographystyle{abbrvnat}
\bibliography{LiteraturAlois}

\end{document}